\documentclass[12pt]{amsart}

\usepackage{amscd,mathrsfs,epsfig,amsthm,amssymb,amsmath}
\usepackage[all]{xy} 
\usepackage[T1]{CJKutf8}
\usepackage[sans]{dsfont}  
\usepackage{hyperref}
\usepackage{tabularx}

\newtheorem{theorem}{Theorem}[subsection]
\newtheorem{corollary}[theorem]{Corollary}

\newtheorem{definition}[theorem]{Definition}
\newtheorem{proposition}[theorem]{Proposition}
\newtheorem{remark}[theorem]{Remark}

\newtheorem{example}[theorem]{Example}

\newtheorem*{thma}{Theorem A}
\newtheorem*{thmb}{Theorem B}

\allowdisplaybreaks[2]

\def\calA{{\mathcal A}}
\def\calB{{\mathcal B}}
\def\Bim{{\calB}im_k}
\def\calC{{\mathcal C}}
\def\calD{{\mathcal D}}

\def\calO{{\mathcal O}}
\def\calP{{\mathcal P}}

\def\calV{{\mathcal V}}
\def\calW{{\mathcal W}}

\def\frakB{{\mathfrak B}}
\def\frakC{{\mathfrak C}}
\def\frakD{{\mathfrak D}}
\def\frakM{{\mathfrak M}}
\def\frakN{{\mathfrak N}}
\def\frakF{{\mathfrak F}}

\def\frakR{{\mathfrak R}}
\def\frakS{{\mathfrak S}}

\def\frakV{{\mathfrak V}}

\def\frakW{{\mathfrak W}}

\def\Vec{{\mathop{2\frakV ect}\nolimits}}
\def\Bif{{\mathop{\frakB if}\nolimits}}
\def\Bim{{\mathop{\frakB im}\nolimits}}
\def\Cat{{\mathop{\frakC at}\nolimits}}

\def\Hom{\mathop{\rm Hom}\nolimits} 
 
\def\Id{\mathop{\rm Id}\nolimits}
\def\lim{\mathop{\varinjlim}\nolimits}
\def\Ob{\mathop{\rm Ob}\nolimits} 
\def\Mor{\mathop{\rm Mor}\nolimits}
\def\Rep{\mathop{\rm Rep}\nolimits}

\def\dom{\mathop{\rm dom}}

\DeclareMathOperator{\rMod}{Mod-}

\begin{document}

\title{Modulated categories and their representations via higher categories}

\author{Fei Xu}
\author{Maoyin Zhang}

\email{fxu@stu.edu.cn}
\email{19myzhang@stu.edu.cn}
\address{Department of Mathematics\\Shantou University\\Shantou, Guangdong 515063, China}



\thanks{The authors \begin{CJK*}{UTF8}{}
\CJKtilde \CJKfamily{gbsn}(徐斐、张懋胤)
\end{CJK*} are partially supported by the NSFC grants No.12171297}

\date{}
\maketitle

\begin{abstract} We consider the 3-category $2\Cat$ whose objects are 2-categories, 1-morphisms are lax functors, 2-morphisms are lax transformations and 3-morphisms are modifications. The aim is to show that it carries interesting representation-theoretic information.

Let $\calC$ be a small 1-category and $\Bim_k$ be the 2-category of bimodules over $k$-algebras, where $k$ is a commutative ring with identity. We call a covariant (resp. contravariant) pseudofunctor from $\calC$ into $\calB im_k$ a modulation (resp. comodulation) on $\calC$, define and study its representations. This framework provides a unified approach to investigate 2-representations of finite groups, modulated quivers and their representations, as well as presheaves of $k$-algebras and their modules. Moreover, several key constructions are natural ingredients in $2\Cat$, and thus it exhibits an interesting application of higher category theory to representation theory.
\end{abstract}

\tableofcontents

\section{Introduction}
Let $\calC$ be a small category. It is said to be \textit{finite} if $\Mor\calC$ is finite, and it is \textit{object-finite} if $\Ob\calC$ is finite. This paper is motivated by the usage of various small categories built on a fixed group \cite{AKO, Ba, DLLX, WX, Xu, XZ}. If $G$ is a finite group, then it can be realized as a finite category with one object. Ganter and Kapranov \cite{GK} introduced the concept of a 2-representation of $G$, which is a pseudofunctor from $G$ into the category of 2-vector spaces $\Vec_k$. One naturally wants to extend this construction to an arbitrary finite category $\calC$. Meanwhile, since $\Vec_k$ is a 2-subcategory of the bimodule category $\Bim_k$, it is appealing to consider pseudofunctors $\calC \to \Bim_k$. In fact, Balmer \cite{Ba} defined some functors from the (finite) orbit category $\calO(G)$ to $\Cat$, which factor through $\calO(G) \to \Bim_k$ (or its variantions), based on the realization of $\Bim_k$ as a 2-subcategory of $\Cat$, the category of small categories (in a suitable universe) \cite[Example 2.2]{GK}. Two other similar instances can be found in \cite[19.5.4]{KS} on Morita equivalences of module stacks, and in \cite{As} on gluing derived equivalences. 

The present paper focuses on both the covariant and contravariant pseudofunctors $\frakM : \calC \to \Bim_k$, trying to demonstrate their roles as a unifying framework to study several existing topics in representation theory and (pre)sheaf theory, including $k$-modulated quivers (where the names come from) and the presheaves of $k$-algebras. After introducing modulated and comodulated categories, we define their representations. From a modulated category (or a comodulated category), we straightforwardly construct an algebra, proving that their modules are identified with the aforementioned representations, under mild assumptions. This generalizes various key results in \cite{DR, Ga, Li, Mi, Si, WX}, with a striking higher categorical formulation.

\begin{thma}
Let $\calC$ be an object-finite category and $\frakM:\calC \to \Bim_k$ be a covariant pseudofunctor. Then its 2-limit, as an object in $\Bim_k$, is a unital associative algebra, called the modulated category algebra and represented by $\frakM[\calC]=\oplus_{\alpha\in\Mor\calC}\frakM(\alpha)$, carrying an equivalence (on the left is simply the category of $A$-$\frakM[\calC]$-bimodules)
$$
\Hom_{\Bim_k}(A,\frakM[\calC])\simeq\Hom_{{\Bim_k}^{\calC}}(\underline{A},\frakM),
$$
where $A$ is an object of $\Bim_k$ (a unital $k$-algebra) and $\underline{A}$ is the corresponding constant pseudofunctor. Particularly, the right module category ${\rm Mod}$-$\frakM[\calC]$ is equivalent to $\Rep_k(\calC,\frakM)$, the category of representations of the modulated category $(\calC,\frakM)$.

Similar constructions and equivalence hold for a comodulation on $\calC$.
\end{thma}

As it shows, many of the above constructions are encoded in higher category theory. The category of all 2-categories is a 3-category, with objects the 2-categories (including 1-categories), 1-morphisms the lax functors (including pseudofunctors), 2-morphisms the lax (natural) transformations, and 3-morphisms the modifications. We offer the following dictionary for the relevant structures defined on an index 1-category $\calC$ in different contexts (here $\underline{k}$ is a constant pseudofunctor) :

\begin{center}
\begin{tabular}{c|c|c|c}
\hline
Structure & Higher Category Th. & Representation Th. & Algebra\\
\hline
\hline
$\frakM : \calC \to \Bim_k$ & a pseudofunctor & a $k$-modulation & an algebra $\frakM[\calC]$\\
\hline
\hline
$\calV : \underline{k} \to \frakM$  & a lax transformation & a $k$-representation & a module\\
\hline
$T : \calV \to \calW$  & a modification  & a morphism & a morphism\\
\hline
\hline
Category & $\Hom(\underline{k},\frakM)$ & $\Rep_k(\calC,\frakM)$ & $\rMod\frakM[\calC]$\\
\hline
\end{tabular}
\end{center}

We emphasize that the 2-representations of Ganter-Kapranov \cite{GK} are covariant pseudofunctors (1-morphisms), while the representations of modulated categories are lax transformations (2-morphisms), traced back to Gabriel \cite{Ga}. These are structures on different levels, but we shall see that the representations of the trivial 2-representation $\underline{k}$ are classical representations (i.e. 1-representations). There is an analogous chart for comodulated categories, which we choose not to write down here. 

If $\frakR$ is a presheaf of $k$-algebras on $\calC$, then it provides both a modulation $\frakM_{\frakR}$ and a comodulation $\frakW_{\frakR}$ on $\calC$. With the comodulation $\frakW_{\frakR}$, one immediately sees that the comodulated category algebra $\frakW_{\frakR}[\calC]$ is the same as the skew category algebra $\frakR[\calC]$ defined in \cite{WX}, and the representations of $(\calC,\frakW_{\frakR})$ are modules of $\frakR$ (in the context of sheaf theory). Therefore, the above result generalizes one of the main theorems in \cite{WX}. Now, there is a geometric way to introduce various finiteness conditions on $\frakR$-modules, via slice categories. Meanwhile, one may characterize the objects of Mod-$\frakW_{\frakR}[\calC]$ by module-theoretic methods. It is natural to ask whether these two seemingly different sets of finiteness conditions agree? We demonstrate that the algebraic finiteness conditions are weaker than the geometric ones.

\begin{thmb} Let $\calC$ be an object-finite category. Given the equivalence $\rMod\frakR\simeq\rMod\frakW_{\frakR}[\calC]$, the finite-type $\frakR$-modules correspond to finitely generated $\frakR[\calC]$-modules, but not vice versa. 
\end{thmb}

In Section 5, we give a small explicit example to tell the differences.

The paper is organized as follows. In Section 2, we review the basics of higher category theory. In Section 3, we introduce modulated and comodulated categories, and their representations, in both higher categorical terms and concrete terms. Then in Section 4 we define the modulated and comodulated algebras, and establish Theorem A by showing that they are certain 2-limits. Finally, we compare algebraic and geometric finiteness conditions on modules over a presheaf of algebras in Section 5.

\section{Basic higher category theory} \label{2cat}

We shall consider the 2-category of bimodules and the pseudofunctors into it. Although they have roots in algebraic geometry, these categorical constructions seem to provide a natural way to organize representation-theoretic information.

All 2-categories form a 3-category, with 2-categories as objects, lax functors as 1-morphisms, lax (natural) transformations as 2-morphisms, and modifications as 3-morphisms. Surprisingly, all of these components carry representation-theoretic information in our case.

\subsection{2-categories} As our central concept, we recall the definition of a 2-category. Classically, a 2-category would be a strict 2-category, but here we follow the trend to mean a \textit{bicategory} (some weak 2-category) introduced by B\'enabou \cite{Be}, see for instance \cite{GK, Lu}.

\begin{definition} A 2-category $\frakC$ consists of the following data
\begin{enumerate}
    \item a class of objects $\Ob\frakC$;

    \item for any pair of objects $x$ and $y$, a category $\Hom_{\frakC}(x,y)$, whose objects are called 1-morphisms;
    
    \item for any $\alpha,\beta\in\Hom_{\frakC}(x,y)$, the morphisms from $\alpha$ to $\beta$ are called 2-morphisms, written as $\alpha \Rightarrow \beta$;

    \item for every triple $x, y, z$, a composition functor
    $$
    -\circ - : \Hom_{\frakC}(y,z)\times\Hom_{\frakC}(x,y) \to \Hom_{\frakC}(x,z);
    $$

    \item for every $x\in\Ob\frakC$, an isomorphism $u_x : 1_x\circ 1_x \Rightarrow 1_x$ in the category $\Hom_{\frakC}(x,x)$, called the unit constraints of $\frakC$;

    \item for every quadruple $w,x,y,z$, an isomorphism from
    $$
    \Hom_{\frakC}(y,z)\times\Hom_{\frakC}(x,y)\times\Hom_{\frakC}(w,x)\to\Hom_{\frakC}(w,z), \ \ (\gamma, \beta, \alpha) \mapsto \gamma\circ (\beta \circ \alpha)
    $$
    to
    $$
    \Hom_{\frakC}(y,z)\times\Hom_{\frakC}(x,y)\times\Hom_{\frakC}(w,x)\to\Hom_{\frakC}(w,z), \ \ (\gamma, \beta, \alpha) \mapsto (\gamma\circ \beta) \circ \alpha,
    $$
    written as $a_{\gamma,\beta,\alpha}: \gamma \circ (\beta \circ \alpha) \Rightarrow (\gamma \circ \beta) \circ \alpha$, called the associativity constraints of $\frakC$;
\end{enumerate}
satisfying the conditions that
\begin{enumerate}
    \item[(i)] for any $x, y\in\Ob\frakC$, the functors $\Hom_{\frakC}(x,y) \to \Hom_{\frakC}(x,y), \alpha \mapsto \alpha\circ 1_x$ and $\Hom_{\frakC}(x,y) \to \Hom_{\frakC}(x,y), \alpha \mapsto 1_y\circ \alpha$ are fully faithful;

    \item[(ii)] for any quadruple of composable morphisms $v {\buildrel{\mu}\over{\to}} w {\buildrel{\alpha}\over{\to}} x {\buildrel{\beta}\over{\to}} y {\buildrel{\gamma}\over{\to}} z$, there is a commutative diagram in $\Hom_{\frakC}(v,z)$
    $$
    \xymatrix{& \gamma \circ ((\beta \circ \alpha) \circ \mu) \ar@{=>}[r] & (\gamma \circ (\beta \circ \alpha)) \circ \mu \ar@{=>}[dr] &\\
    \gamma \circ (\beta \circ (\alpha \circ \mu)) \ar@{=>}[ur] \ar@{=>}[dr] &&& ((\gamma \circ \beta) \circ \alpha) \circ \mu\\
    & (\gamma \circ \beta) \circ (\alpha \circ \mu) \ar@{=>}[urr] &&}
    $$
\end{enumerate}
\end{definition}

We often omit the symbol $\circ$ and write $\beta\alpha$ for $\beta\circ\alpha$. There are a horizontal composition $-\star-$, as well as a vertical composition $- \cdot -$, of 2-morphisms. An invertible 2-morphism will be called a \textit{2-isomorphism}. If the associativity constraints are identities, the 2-category is said to be \textit{strict}.

The category $\Cat$ of all small categories (in a suitable universe) is the prototype of 2-categories. Its objects are categories, 1-morphisms are functors and 2-morphisms are natural transformations. It is strict.

Let $\calC$ be a 1-category. It can be regarded as a 2-category by asking the 2-morphisms to be the identities on 1-morphisms. We sometimes refer to such a 1-category an \textit{index category}.

We shall focus on the following (non-strict) 2-category.

\begin{example}
Let $k$ be a commutative ring with identity. We consider the 2-category $\Bim_k$ of bimodules:
\begin{enumerate}
    \item the objects are unital associative $k$-algebras $A,B$ etc;

    \item the 1-morphisms from $A$ to $B$ are $A$-$B$-bimodules $_{A}M_{B} : A \to B$;

    \item the 2-morphisms are bimodule homomorphisms $_{A}M_{B}\Rightarrow _{A}N_{B}$,
\end{enumerate}
in which the composite of 1-morphisms $A {\buildrel{M}\over{\to}} B {\buildrel{N}\over{\to}} C$ is given by $N\circ M := M\otimes_BN$, from $A$ to $C$. The identity 1-morphism $1_A$ is exactly $A$.

The 2-subcategory $\Vec_k$, of 2-vector spaces, consists of objects $A=k^n$, for all $n\ge 1$.
\end{example}

\begin{remark} Based on the concept of a lax functor, the 2-category $\Bim_k$ may be realized as a 2-subcategory of $\Cat$, by assigning an algebra $A$ to its category of right modules $\rMod{A}$, and a 1-morphism $M: A \to B$ to $-\otimes_AM : \rMod{A} \to \rMod{B}$. This, along with its variations, including $\calD\Bim_k$, were considered by \cite{As, Ba}.

We shall not use this realization in this article.
\end{remark}

\subsection{Pseudofunctors}

Here we follow Johnson-Yau \cite{JY} and Lurie \cite{Lu}. We refrain from introducing general lax functors between 2-categories, which we will not use here.

Let $\calC$ be a 1-category. We shall only consider unitary lax functors from $\calC$ to particular 2-categories, especially $\Cat$ and $\Bim_k$. By \cite[Remark 2.2.2.49]{Lu}, it is harmless to assume unitary lax functors to be strictly unitary. 

In this article, we shall focus on a special class of lax functors, namely the {\it strictly unitary pseudofunctors}, defined on a 1-category $\calC$ \cite{JY}, because the motivating examples, including modulations on a quiver and presheaves of $k$-algebras on a category, are of this type. For brevity, we shall abbreviate the terminology to just {\it pseudofunctors}. 

\begin{definition}\label{cpf} Let $\calC$ be a small 1-category and $\frakC$ be a 2-category. A covariant pseudofunctor $\frakS: \calC \to \frakC$ consists of the following data
\begin{enumerate}
    \item for any $x\in\Ob\calC$, an object $\frakS(x)\in\Ob\frakC$,

    \item for any morphism $\alpha:x\to y$, a 1-morphism $\frakS(\alpha) : \frakS(x) \to \frakS(y)$,

    \item for any morphisms $\alpha : x \to y$ and $\beta : y \to z$, a 2-isomorphism $c_{\beta,\alpha} : \frakS(\beta)\frakS(\alpha) \Rightarrow \frakS(\beta\alpha)$,
\end{enumerate}
    satisfying the conditions
    \begin{enumerate}
        \item[(i)] $\frakS(1_x)=\Id_{\frakS(x)}$,

        \item[(ii)] for any $\mu:w \to x$, $\alpha : x \to y$ and $\beta : y \to z$, a commutative diagram 
        $$
\xymatrix{\frakS(\beta)\frakS(\alpha)\frakS(\mu) \ar[rr]^{c_{\beta,\alpha}\star\Id_{\frakS(\mu)}} \ar[d]_{\Id_{\frakS(\beta)}\star c_{\alpha,\mu}} & & \frakS(\beta\alpha)\frakS(\mu) \ar[d]^{c_{\beta\alpha,\mu}}\\
\frakS(\beta)\frakS(\alpha\mu) \ar[rr]_{c_{\beta,\alpha\mu}} & & \frakS(\beta\alpha\mu).}
        $$
    \end{enumerate}
\end{definition}

We shall consider contravariant pseudofunctors as well, so we state its definition for future reference.

\begin{definition}\label{ctpf} Let $\calC$ be a small 1-category. A contravariant pseudofunctor $\frakS: \calC \to \frakC$ consists of the following data
\begin{enumerate}
    \item for any $x\in\Ob\calC$, an object $\frakS(x)\in\frakC$,

    \item for any morphism $\alpha:x\to y$, a 1-morphism $\frakS(\alpha) : \frakS(y) \to \frakS(x)$,

    \item for any morphisms $\alpha : x \to y$ and $\beta : y \to z$, a 2-isomorphism $c_{\alpha,\beta} : \frakS(\alpha)\frakS(\beta) \Rightarrow \frakS(\beta\alpha)$,
\end{enumerate}
    satisfying the conditions
    \begin{enumerate}
        \item[(i)] $\frakS(1_x)=\Id_{\frakS(x)}$,

        \item[(ii)] for any $\mu:w \to x$, $\alpha : x \to y$ and $\beta : y \to z$, a commutative diagram 
        $$
\xymatrix{\frakS(\mu)\frakS(\alpha)\frakS(\beta) \ar[rr]^{\Id_{\frakS(\mu)}\star c_{\alpha,\beta}} \ar[d]_{c_{\mu,\alpha}\star\Id_{\frakS(\beta)}} & & \frakS(\mu)\frakS(\beta\alpha) \ar[d]^{c_{\mu,\beta\alpha}}\\
\frakS(\alpha\mu)\frakS(\beta) \ar[rr]_{c_{\alpha\mu,\beta}} & & \frakS(\beta\alpha\mu).}
        $$
    \end{enumerate}
\end{definition}

Throughout this paper, we shall insist on the strictly unitary condition as it is convenient for our applications in representation theory.

\begin{remark} Let $\frakM : \calC \to \frakC$ be a pseudofunctor. Suppose $c$ is an object in $\Ob\frakC$. One can define a \textit{constant pseudofunctor} $\underline{c} : \calC \to \frakC$, sending every $x\in\Ob\calC$ to $c$ and every $\alpha\in\Mor\calC$ to $1_c$ (the choices of $c_{\beta,\alpha}$'s are clear). Later on, we shall deal with the constant pseudofunctor $\underline{k} : \calC \to \Bim_k$, which sends every object to $k$ and every morphism of $\calC$ to the bimodule $k$. 

General lax (and colax) functors are also studied in representation theory, see for example \cite{As}.
\end{remark}

\subsection{Lax (natural) transformations}

We often need to compare two pseudofunctors.

\begin{definition} Let $\frakM, \frakN : \calC \to \frakC$ be two pseudofunctors. A lax (natural) transformation $\calV : \frakM \to \frakN$ consists of the following data
\begin{enumerate}
\item for each $x \in \Ob\calC$, a 1-morphism $\calV(x) : \frakM(x) \to \frakN(x)$,

\item for each $\alpha : x \to y$, a 2-morphism $\calV(\alpha): \frakN(\alpha)\circ\calV(x) \Rightarrow \calV(y)\circ\frakM(\alpha)$ in $\frakC$ as follows

$$
\xymatrix{\frakM(x) \ar[rr]^{\frakM(\alpha)} \ar[d]_{\calV(x)} & & \frakM(y) \ar[d]^{\calV(y)} \\
\frakN(x) \ar[rr]_{\frakN(\alpha)} \ar@{=>}[urr] & & \frakN(y)}
$$
\end{enumerate}
satisfying the following commutative diagrams:
\begin{enumerate}
\item (lax unity) 
$$
\xymatrix{\Id_{\frakN(x)}\circ\calV(x) \ar[r] \ar[d]_{=} & \calV(x) \ar[r] & \calV(x)\circ\Id_{\frakM(x)} \ar[d]^{=}\\
\frakN(1_x)\circ\calV(x) \ar[rr]_{\calV(1_x)} && \calV(x)\circ\frakM(1_x),}
$$
\item (lax naturality) and for every $\alpha : x \to y$ and $\beta : y \to z$,
$$
\xymatrix{\frakN(\beta)\circ(\calV(y)\circ\frakM(\alpha)) \ar[r] & (\frakN(\beta)\circ\calV(y))\circ\frakM(\alpha) \ar[r] & (\calV(z)\circ\frakM(\beta))\circ\frakM(\alpha) \ar[d]\\
\frakN(\beta)\circ(\frakN(\alpha)\circ\calV(x)) \ar[u] && \calV(z)\circ(\frakM(\beta)\circ\frakM(\alpha)) \ar[d] \\
(\frakN(\beta)\circ\frakN(\alpha))\circ\calV(x) \ar[u] \ar[r] & \frakN(\beta\alpha)\circ\calV(x) \ar[r] & \calV(z)\circ\frakM(\beta\alpha).}
$$
\end{enumerate}
\end{definition}

Natural transformations between two strict functors between 1-categories are lax transformations.

\begin{remark}
\begin{enumerate}
\item If the 2-morphism in (2) of the above definition is invertible, then such a lax transformation is called {\rm strong}. In the theory of fibred categories over $\calC$ \cite{Stack}, only strong transformations are taken into account.

\item If the 2-morphism in (2) is reversed, then the transformation is called {\rm oplax}. It would be interesting to learn its role in representation theory.
\end{enumerate}
\end{remark}

\subsection{Modifications} The category of all pseudofunctors from $\calC$ to $\frakC$ is itself a 2-category, with objects pseudofunctors, and 1-morphisms lax transformations. The 2-morphisms are called \textit{modifications}, see for instance \cite{JY}.

\begin{definition}\label{modi} Let $\frakM$ and $\frakN$ be two pseudofunctors $\calC \to \frakC$, and $\calV, \calW : \frakM \to \frakN$ be two lax transformations. A modification $T : \calV \to \calW$ assigns to each $x\in\Ob\calC$ a 2-morphism $T_x : \calV(x) \Rightarrow \calW(x)$ such that for any $\alpha : x \to y$ in $\calC$, the following diagram commutes
$$
\xymatrix{\frakN(\alpha)\circ\calV(x) \ar@{=>}[rr]^{\Id_{\frakN(\alpha)}\star T_x} \ar@{=>}[d]_{\calV(\alpha)} & & \frakN(\alpha)\circ\calW(x) \ar@{=>}[d]^{\calW(\alpha)}\\
\calV(y)\circ\frakM(\alpha) \ar@{=>}[rr]_{T_y\star\Id_{\frakM(\alpha)}} & & \calW(y)\circ\frakM(\alpha).}
$$
\end{definition}

The category of all 2-categories is a 3-category, with objects 2-categories, 1-morphisms lax functors, 2-morphisms lax transformations and 3-morphisms modifications. It is common to use $\Hom(\frakC,\frakD)$ for the set of 1-morphisms between two 2-cateories, $2\Hom(\frakM,\frakN)$ for the set of lax transformations between two lax functors, and $3\Hom(\calV,\calW)$ for the set of modifications between transformations. We write $\Hom_{\frakC^{\calC}}(\frakM,\frakN)$ for the hom category consisting of relevant 2- and 3-moprhisms, as objects and morphism.

\subsection{2-limits} We recall from \cite{JY} the concept of a lax bilimit and, being consistent and concise, here we will call it a 2-limit. Let $\calC$ be a 1-category and $\frakC$ be a 2-category. For brevity, we write $\frakC^{\calC}$ for the 2-category, whose objects are lax functors from $\calC$ to $\frakC$, 1-morphisms are lax transformations and 2-morphisms are modifications. Given two peudofunctors $\frakM,\frakN : \calC \to \frakC$, $\Hom_{\frakC^{\calC}}(\frakM,\frakN)$ is the hom category.

\begin{definition} Let $\frakM : \calC \to \frakC$ be a pseudofunctor. Suppose $\Ob\calC$ is a set. A 2-limit of $\frakM$ in $\frakC$ is a pair $(l,\pi)$, where $l$ is an object in $\frakC$ and $\pi : \underline{l} \to \frakM$ is a lax transofrmation, such that for each object $c$ in $\Ob\frakC$, there is a category equivalence
$$
\pi_* : \Hom_{\frakC}(c,l){\buildrel{\simeq}\over{\to}}\Hom_{\frakC^{\calC}}(\underline{c},\frakM),
$$
induced by $\pi$.
\end{definition}

Be aware that this is the 2-limit of a pseudofunctor of $\frakM$, not its pseudo 2-limit, see \cite[Definition 5.1.1]{JY} for the difference. It is well-known that the 2-limit of $\frakM$ is unique up to equivalence, if it exists.

We shall demonstrate shortly that all the above abstract categorical constructions encode important representation-theoretic information.

\section{Modulated categories and their representations} \label{mc}

In this section, we demonstrate how a pseudofunctor organizes interesting information, as a generalization of the modulated quiver introduced by P. Gabriel \cite{Ga}, investigated and generalized by V. Dlab, C. M. Ringel \cite{DR}, D. Simson \cite{Si} and F. Li \cite{Li}, among many representation theorists. 

Our work stems from the 2-representation theory of finite groups \cite{GK}, but includes the modulated quivers (species) and their representations \cite{DR, Ga, Li, Si}, as well as the presheaves of $k$-algebras and their modules \cite{Stack, KS, WX}.

\subsection{Modulations and comodulations}

By definition, a pseudofunctor is an assignment from an index 1-category $\calC$ to a 2-category $\frakC$. From now on, we shall focus on pseudofunctors with values in $\Bim_k$. 

\begin{definition} Let $\calC$ be a small category. 
\begin{enumerate}
\item A (k-)modulation of $\calC$ is a covariant pseudofunctor $\frakM : \calC \to \Bim_k$, and  $(\calC,\frakM)$ is called a (k-)modulated category.

\item A (k-)comodulation of $\calC$ is a contravariant pseudofunctor $\frakW : \calC \to \Bim_k$, and $(\calC,\frakW)$ is called a (k-)comodulated category.
\end{enumerate}
\end{definition}

Be aware that we then have $\frakS(1_x)=\Id_{\frakS(x)}=\frakS(x)$, as a $\frakS(x)$-$\frakS(x)$-bimodule, $\forall x \in\Ob\calC$. Under the circumstance, the commutative diagram in Definition \ref{cpf} reads as follows: for a covariant pseudofunctor and for any $\mu:w \to x$, $\alpha : x \to y$ and $\beta : y \to z$, a commutative diagram 
$$
\xymatrix{\frakS(\mu)\otimes_{\frakS(x)}\frakS(\alpha)\otimes_{\frakS(y)}\frakS(\beta) \ar[rrr]^{\Id_{\frakS(\mu)}\otimes c_{\beta,\alpha}} \ar[d]_{c_{\alpha,\mu}\otimes\Id_{\frakS(\beta)}} & & & \frakS(\mu)\otimes_{\frakS(x)}\frakS(\beta\alpha) \ar[d]^{c_{\beta\alpha,\mu}} \\
\frakS(\alpha\mu)\otimes_{\frakS(y)}\frakS(\beta) \ar[rrr]_{c_{\beta,\alpha\mu}} & & & \frakS(\beta\alpha\mu),}
$$    
    and for contravariant one in Definition \ref{ctpf}
$$
\xymatrix{\frakS(\beta)\otimes_{\frakS(y)}\frakS(\alpha)\otimes_{\frakS(x)}\frakS(\mu) \ar[rrr] \ar[d] & & & \frakS(\beta\alpha)\otimes_{\frakS(x)}\frakS(\mu) \ar[d] \\
\frakS(\beta)\otimes_{\frakS(y)}\frakS(\alpha\mu) \ar[rrr] & & & \frakS(\beta\alpha\mu).}
$$

Since any quiver generates a free category, the modulation generalizes the concept of a modulated quiver, or a species, where $\frakM(x)$ is only assumed to be a division ring \cite{Ga, DR,Si} in the original setting (later extended to arbitrary algebras in \cite{Li} under the term of a pseudo-modulation). 

We are particularly interested in the following case in which the formulation is not seen in the literature.

\begin{example}\label{presheaf}
Let $\calC$ be a small category. A presheaf of $k$-algebras $\frakR$ gives rise to a modulation on $\calC$ by setting a bimodule along with the algebra homomorphism $\frakR(\alpha) : \frakR(y)\to\frakR(x)$
$$
\frakM_{\frakR}(\alpha)= _{\frakR(x)}\frakR(x)_{\frakR(y)} : \frakR(x) \to \frakR(y), \forall \alpha :x \to y.
$$

In fact, the presheaf of algebras also provides a comodulation, via
$$
\frakW_{\frakR}(\alpha)=_{\frakR(y)}\frakR(x)_{\frakR(x)} : \frakR(y) \to \frakR(x), \forall \alpha :x \to y.
$$ 
This comodulation, along with its representations, was studied in \cite{GS88, WX}, for example.
\end{example}

Upon the realization of $\Bim_k$ inside $\Cat$, the modulation affords the restriction $\rMod\frakR(x) \to \rMod\frakR(y)$, along $\frakR(\alpha)$, while the comodulation leads to the induction $\rMod\frakR(y) \to \rMod\frakR(x)$. We may pass from here to the work of  Balmer \cite{Ba}.

\begin{remark} It is possible to replace $\Bim_k$ by a larger 2-category of bifunctors $\Bif_k$, whose objects are (small) k-linear categories, 1-morphisms are additive bifunctors, and 2-morphisms are natural transformations of additive bifunctors. This was used by Asashiba \cite{As}. Since this doesn't pose much more difficulties, at least in our situation, but requires higher technicalities in terms of terminologies, we refrain from pursing this generality. 
\end{remark}

\subsection{Representations}

In the sense of Ganter-Kapranov \cite{GK}, a pseudofuctor $\frakM : \calC \to \Vec_k$ would be called a 2-representation of $\calC$. However, it is different from the notion that we will define next. 

\begin{definition} A representation of the modulated quiver $(\calC,\frakM)$ is a lax transformation $\calV : \underline{k} \to \frakM$. Denote by $\Rep_k(\calC,\frakM)=\Hom_{\Bim_k^{\calC}}(\underline{k},\frakM)$, and call it the category of representations.
\end{definition}

The category $\Rep_k(\calC,\frakM)=\Hom_{\Bim_k^{\calC}}(\underline{k},\frakM)$ is what Ganter-Kapranov \cite{GK} called the category of $G$-equivariant objects. If we unwrap the abstract concept, the definition is rewritten in the following equivalent form.

\begin{definition}\label{repmc}
A representation $\calV$ of the modulated category $(\calC,\frakM)$ consists of the following data
\begin{enumerate}
    \item for any $x\in\Ob\calC$, a right $\frakM(x)$-module $\calV(x)$, and

    \item for any $\alpha : x \to y$, a morphism of right $\frakM(y)$-modules 
    $$
    \calV(\alpha) : \calV(x)\otimes_{\frakM(x)}\frakM(\alpha) \to \calV(y) ,
    $$
\end{enumerate}
satisfying the condition that there is a commutative diagram, for any $\alpha : x \to y$ and $\beta: y \to z$, 
$$
\xymatrix{\calV(x)\otimes_{\frakM(x)}\frakM(\alpha)\otimes_{\frakM(y)}\frakM(\beta) \ar[d]_{\Id_{\calV(x)}\otimes c_{\alpha,\beta}} \ar[rrr]^{\calV(\alpha)\otimes\Id_{\frakM(\beta)}} & & & \calV(y)\otimes_{\frakM(y)}\frakM(\beta) \ar[d]^{\calV(\beta)} \\
\calV(x)\otimes_{\frakM(x)}\frakM(\beta\alpha) \ar[rrr]_{\calV(\beta\alpha)}  & & & \calV(z).}
$$
\end{definition}

Note that $\calV(\alpha)$ is a $\underline{k}(x)$-$\frakM(y)$-bimodule homomorphism, thus a 2-morphism in $\Bim_k$. A representation looks like a covariant functor, but it does not supply a reasonable morphism from $\calV(x)$ to $\calV(y)$ for every $\alpha : x \to y$, in general. 

\begin{remark}
\begin{enumerate}
\item If $\calC$ is the free category generated by a quiver, then the above definition coincides with the definition of a representation of a modulated quiver \cite{Ga, DR}, which is a cornerstone in algebra representation theory. 

\item By the standard adjunction between $\otimes$ and $\Hom$, a $\frakM(y)$-module homomorphism $\calV(\alpha)$ corresponds to a $\frakM(x)$-homomorphism of right modules, still written as $\calV(\alpha)$,
$$
\calV(\alpha) : \calV(x) \to \Hom_{\frakM(x)}(\frakM(\alpha) ,\calV(y)).
$$

\end{enumerate}
\end{remark}

\begin{definition} Let $\calV, \calW : \underline{k} \to \frakM$ be two representations of the modulated category $(\calC,\frakM)$. Then we call a modification $T : \calV \to \calW$ a morphism of representations.
\end{definition}

Again, we may interpret the concept as follows. Based on Definition \ref{modi}, a morphism $T$ consists of $\underline{k}(x)$-$\frakM(x)$-bimodule homomorphism $T_x : \calV(x) \to \calW(x)$, $\forall x\in\Ob\calC$, satisfying the following commutative diagram 
$$
\xymatrix{\calV(x)\otimes_{\frakM(x)}\frakM(\alpha)=\frakM(\alpha)\circ\calV(x) \ar@{=>}[rr]^{\Id_{\frakM(\alpha)}\star T_x} \ar@{=>}[d]_{\calV(\alpha)} & & \frakM(\alpha)\circ\calW(x)=\calW(x)\otimes_{\frakM(x)}\frakM(\alpha) \ar@{=>}[d]^{\calW(\alpha)}\\
\underline{k}(\alpha)\otimes_{\underline{k}(y)}\calV(y)=\calV(y)\circ\underline{k}(\alpha) \ar@{=>}[rr]_{T_y \star\Id_{\underline{k}(\alpha)}} & & \calW(y)\circ\underline{k}(\alpha)=\underline{k}(\alpha)\otimes_{\underline{k}(y)}\calW(y).}
$$

Therefore, we obtain the following equivalent definition, which is the one used in the literature \cite{Ga, DR}.

\begin{definition} Let $\calV, \calW$ be two representations of the modulated category $(\calC,\frakM)$. A morphism $T : \calV \to \calW$ consists of a set of $\frakM(x)$-homomorphisms 
$$
T_x : \calV(x) \to \calW(x),
$$ 
satisfying the condition that $T_y\calV(\alpha)=\calW(\alpha)(T_x\otimes\Id_{\frakM(\alpha)})$, for any $\alpha : x \to y$. 
\end{definition}

From here, one can easily define the kernel and cokernel of a morphism. It is straightforward to see that $\Rep_k(\calC,\frakM)$, the category of all representations of $(\calC,\frakM)$, is abelian. This is known for modulated quivers.

For future reference, we also spell out the definition of a representation of a comodulated category $(\calC,\frakW)$. A representation is a lax transformation $\calV : \underline{k} \to \frakW$ and a morphism of representations is a modification $T : \calV \to \calW$. They have the following equivalent characterizations.

\begin{definition}\label{CMR}
A representation $\calV$ of the comodulated category $(\calC,\frakW)$ consists of the following data
\begin{enumerate}
    \item for any $x\in\Ob\calC$, a right $\frakW(x)$-module $\calV(x)$, and

    \item for any $\alpha : x \to y$, a morphism of right $\frakW(x)$-modules 
    $$
    \calV(\alpha) : \calV(y)\otimes_{\frakW(y)}\frakW(\alpha) \to \calV(x) ,
    $$
\end{enumerate}
satisfying the condition that there is a commutative diagram, for any $\alpha : x \to y$ and $\beta: y \to z$, 
$$
\xymatrix{\calV(z)\otimes_{\frakW(z)}\frakW(\beta)\otimes_{\frakW(y)}\frakW(\alpha) \ar[d]_{\Id_{\calV(z)}\otimes c_{\alpha,\beta}} \ar[rrr]^{\calV(\beta)\otimes\Id_{\frakW(\alpha)}} & & & \calV(y)\otimes_{\frakW(y)}\frakW(\alpha) \ar[d]^{\calV(\alpha)} \\
\calV(z)\otimes_{\frakW(z)}\frakW(\beta\alpha) \ar[rrr]_{\calV(\beta\alpha)}  & & & \calV(x).}
$$
\end{definition}

The above representation also does not generate a contravariant functor. However, when $\frakW_{\frakR}$ is given by a presheaf of algebras $\frakR$, condition (2) is equivalent to giving a set of $\frakR(y)$-homomoprhisms $\calV(y) \to \calV(x)$, via the usual adjunction. This makes $\calV$ to be a right $\frakR$-module, which was studied in algebraic geometry \cite{Stack, KS} and in representation theory \cite{WX}.

\begin{remark}
When $\frakR$ is a presheaf of $k$-algebras on $\calC$ which provides a comodulation $\frakW_{\frakR}$ by $\frakW_{\frakR}(\alpha)=_{\frakR(y)}\frakR(x)_{\frakR(x)}$, a $\frakW_{\frakR}(x)$-homomorphism (i.e. an $\frakR(x)$-homomorphism)
$$
\calV(\alpha) : \calV(y)\otimes_{\frakW(y)}\frakW_{\frakR}(\alpha)=\calV(y)\otimes_{\frakR(y)}\frakR(x)\to\calV(x)
$$ 
is identified with a $\frakW_{\frakR}(y)$-homomorphism (i.e. an $\frakR(y)$-homomorphism) $\calV(y) \to \calV(x)$, via the usual adjunction between the induction and restriction along $\frakR(\alpha) : \frakR(y) \to \frakR(x)$. Therefore, a representation $\calV$ of the modulated category $(\calC,\frakW_{\frakR})$ becomes a presheaf of $k$-modules, with a right $\frakR$-module structure. This is the case studied in \cite{WX}.
\end{remark}

\begin{definition} Let $\calV, \calW$ be two representations of the comodulated category $(\calC,\frakW)$. A morphism $\Phi : \calV \to \calW$ consists of a set of $\frakW(x)$-homomorphisms 
$$
\Phi(x) : \calV(x) \to \calW(x),
$$ 
satisfying the condition $\Phi(x)\calV(\alpha)=\calW(\alpha)(\Phi(y)\otimes\Id_{\frakW(\alpha)})$, for any $\alpha : x \to y$.    
\end{definition}

All representations of $(\calC,\frakW)$ also form an abelian category $\Rep_k(\calC,\frakW)$, which is the same as $\Hom_{\Bim_k^{\calC^{op}}}(\underline{k},\frakW)$.

\section{Modulated category algebras} \label{mca}

Suppose $\calC$ is a small category. A pseudofunctor immediately offers an algebra that is of great interest. Consider the 2-limit of $\frakM : \calC \to \Bim_k$. If it exists, then it is a pair $(L,\pi)$, where $L$ is a unital associative $k$-algebra (an object of $\Ob\Bim_k$) and $\pi : \underline{L} \to \frakM$ is a lax transformation, equipped with a category equivalence, induced by $\pi$,
$$
\Hom_{\Bim_k}(A,L) \simeq \Hom_{\Bim_k^{\calC}}(\underline{A},\frakM),
$$
for every $A \in \Ob\Bim_k$. Note that $\Hom_{\Bim_k}(A,L)$ is simply the category of $A$-$L$-bimodules. When $A=k$, it reduces to the category of right $L$-modules.

The upshot is that we can explicitly construct the 2-limit $L$ of $\frakM : \calC \to \Bim_k$.

\subsection{Modulated and comodulated category algebras}

\begin{definition} Let $(\calC,\frakM)$ be a modulated category. The modulated category algebra $\frakM[\calC]$ is defined as
$$
\bigoplus_{\alpha\in\Mor\calC}\frakM(\alpha),
$$
in which the multiplication of $m_{\alpha}\in\frakM(\alpha)$ and $m_{\beta}\in\frakM(\beta)$ is given by the structure isomorphisms of the modulation
    	\begin{eqnarray}
		m_{\alpha}\ast m_{\beta}=
		\begin{cases}
			c_{\beta,\alpha}(m_{\alpha}\otimes m_{\beta}) \in \frakM(\beta\alpha),       & \text{if} ~{\rm dom}(\beta)={\rm cod}(\alpha); \notag \\
			0, & {\rm otherwise}.
		\end{cases}
	\end{eqnarray} 
\end{definition}

Note that $c_{\beta,\alpha} : \frakM(\beta)\circ\frakM(\alpha)=\frakM(\alpha)\otimes\frakM(\beta)\to \frakM(\beta\alpha)$ is an isomorphism for any $\alpha : x\to y$ and $\beta : y \to z$.

\begin{remark}
    \begin{enumerate}
        \item The algebra $\frakM[\calC]$ is associative.
        
        \item If $\Ob\calC$ is finite, then the algebra $\frakM[\calC]$ has an identity $\sum_{x\in\Ob\calC}1_{\frakM(x)}$, where $1_{\frakM(x)}$ is the identity of $\frakM(x)=\frakM(1_x)$.

        \item If $\Mor\calC$ is finite and every $\frakM(\alpha)$ is a finite-dimensional $k$-module, then the algebra $\frakM[\calC]$ is a finite-dimensional $k$-algebra.

        \item It is not clear if $\frakM[\calC]$ admits an alternative formation via the Grothendieck construction as in \cite{WX}.
    \end{enumerate}
\end{remark}

For future reference, here is the comodulated version.

\begin{definition} Let $(\calC,\frakW)$ be a comodulated category. The comodulated category algebra $\frakW[\calC]$ is defined as
$$
\bigoplus_{\alpha\in\Mor\calC}\frakW(\alpha),
$$
in which the multiplication of $w_{\alpha}\in\frakW(\alpha)$ and $w_{\beta}\in\frakW(\beta)$ is given by the structure isomorphisms of the comodulation
    	\begin{eqnarray}
		w_{\beta}\ast w_{\alpha}=
		\begin{cases}
			c_{\alpha,\beta}(w_{\beta}\otimes w_{\alpha})\in\frakW(\beta\alpha),       & \text{if} ~{\rm dom}(\beta)={\rm cod}(\alpha); \notag \\
			0, & {\rm otherwise}.
		\end{cases}
	\end{eqnarray} 
\end{definition}

Note that $c_{\alpha,\beta} : \frakW(\alpha)\circ\frakW(\beta)=\frakW(\beta)\otimes\frakW(\alpha)\to \frakW(\beta\alpha)$ is an isomorphism for any $\alpha : x\to y$ and $\beta : y \to z$.

\begin{definition} Suppose $\Ob\calC$ is finite. Given $\frakM : \calC \to \Bim_k$, one can construct a lax transformation $\pi : \underline{\frakM[\calC]} \to \frakM$ such that
\begin{enumerate}
\item for each $x\in\Ob\calC$, $\pi(x) : \frakM[\calC] \to \frakM(x)=\frakM(1_x)$ is $\frakM[\calC]1_{\frakM(x)}$, a $\frakM[\calC]$-$\frakM(x)$-bimodule; (here $1_{\frakM(x)}$ is the identity of $\frakM(x)$.)

\item for each $\alpha : x \to y$ in $\Mor\calC$, 
$$
\pi(\alpha) : \frakM[\calC]1_{\frakM(x)}\otimes_{\frakM(x)}\frakM(\alpha) \to \frakM[\calC]\otimes_{\frakM[\calC]}\frakM[\calC]1_{\frakM(y)}
$$ 
is given by the multiplication, a canonical $\frakM[\calC]$-$\frakM(y)$-bimodule homomorphism.
\end{enumerate}
\end{definition}

There is an analogous lax transformation $\pi : \underline{\frakW[\calC]} \to \frakW$ for a comodulation $\frakW$ on an object-finite category $\calC$.

\subsection{2-limits and module categories} When $\Ob\calC$ is finite, $\frakM[\calC]$ becomes an object of $\Bim_k$. We shall prove that $(\frakM[\calC],\pi)$ is the 2-limit of $\frakM : \calC \to \Bim_k$. Then some statements about module categories will follow.  

\begin{theorem} Let $A$ be an object of $\Bim_k$. Soppose $\Ob\calC$ is finite, and $\frakM : \calC \to \Bim_k$ is a pseudofunctor. Then $\pi$ induces a functor
$$
\pi_* : \Hom_{\Bim_k}(A,\frakM[\calC]) \to \Hom_{\Bim_k^{\calC}}(\underline{A},\frakM),
$$
which becomes an equivalence for each $A\in\Ob\Bim_k$.
\end{theorem}

By definition, $\Hom_{\Bim_k}(A,\frakM[\calC])$ is exactly the category of $A$-$\frakM[\calC]$-bimodules.

\begin{proof} Firstly we describe $\pi_*$. There is a standard construction in general situation, see \cite[Proposition 5.1.10]{JY}. Let $M : A \to \frakM[\calC]$ be an object in $\Hom_{\Bim_k}(A,\frakM[\calC])$, that is, a $A$-$\frakM[\calC]$-bimodule. Then it gives rise to a lax transformation $\underline{M} : \underline{A} \to \underline{\frakM[\calC]}$, such that $\underline{M}_x=M, \forall x\in\Ob\calC$, and $\underline{M}(\alpha)= M\otimes_{\frakM[\calC]}\frakM[\calC] \to A\otimes_A M, \forall \alpha :x \to y$, is given by the identity map on the bimodule $M$. Composing with $\pi$ it results in a lax transformation $\pi_*(M)=\pi\star\underline{M}$, that is, an object of $\Hom_{\Bim_k^{\calC}}(\underline{A},\frakM)$. 

If $\phi : M \to N$ is a morphism in $\Hom_{\Bim_k}(A,\frakM[\calC])$, that is, a $A$-$\frakM[\calC]$-bimodule homomorphism, then it yields a modification $\underline{M} \to \underline{N}$, based on which one builds a modification $T=\pi_*(\phi) : \pi_*(\underline{M}) \to \pi_*(\underline{N})$, that is, a morphism in $\Hom_{\Bim_k^{\calC}}(\underline{A},\frakM)$. At each $x\in\Ob\calC$, 
$$
T_x : M\otimes_{\frakM[\calC]}\frakM[\calC]1_{\frakM(x)} \to N\otimes_{\frakM[\calC]}\frakM[\calC]1_{\frakM(x)}
$$ 
is the $A$-$\frakM(x)$-bimodule homomorphism $\phi\otimes\Id_{\frakM[\calC]1_{\frakM(x)}}$.

Secondly, if $\calV$ is an object of $\Hom_{\Bim_k^{\calC}}(\underline{A},\frakM)$, that is, a lax transformation, then $\calV(x) : A \to \frakM(x)$ is a $A$-$\frakM(x)$-bimodule, for all $x$, and 
$$
\calV(\alpha): \calV(x)\otimes_{\frakM(x)}\frakM(\alpha) \to \calV(y)
$$ 
is a $A$-$\frakM(y)$-bimodule homomorphism, for all $\alpha : x \to y$. It leads to a natural $A$-$\frakM[\calC]$-bimodule structure on $M_{\calV}=\oplus_{x\in\Ob\calC}\calV(x)$ as follows. For any $a\in A$, $v_x\in\calV(x)$ and $m_{\alpha}\in\frakM(\alpha)$, we set
    	\begin{eqnarray}
		av_x m_{\alpha} =
		\begin{cases}
			\calV(\alpha)(av_x\otimes m_{\alpha}),       & \text{if} ~\dom(\alpha)=x; \notag \\
			0, & {\rm otherwise}.
		\end{cases}
	\end{eqnarray} 
Assume $a,b \in A$, $\alpha: x\to y$ and $\beta:y\to z$ to be two composable morphisms. If $m_{\alpha}\in\frakM(\alpha)$ and $m_{\beta}\in\frakM(\beta)$, then, by the definition of a lax transformation,
$$
b(av_x m_{\alpha}) m_{\beta}=\calV(\beta)\{b[\calV(\alpha)(av_x\otimes m_{\alpha})]\otimes m_{\beta}\}=ba\calV(\beta\alpha)(v_x\otimes c_{\beta,\alpha}(m_{\alpha}\otimes m_{\beta})),
$$
but the rightmost term is $ba\calV(\beta\alpha)(v_x\otimes(m_{\alpha}*m_{\beta}))=(ba)v_x(m_{\alpha}* m_{\beta})$, with $m_{\alpha}*m_{\beta}\in\frakM(\beta\alpha)$. This makes $M_{\calV}$ an object of $\Hom_{\Bim_k}(A,\frakM[\calC])$. 

Suppose $T : \calV \to \calW$ is a modification, that is, a morphism in $\Hom_{\Bim_k^{\calC}}(\underline{A},\frakM)$. Then $T_x : \calV(x) \to \calW(x)$ is a $A$-$\frakM(x)$-bimodule homomorphism, for all $x$, and $T_y\calV(\alpha)=\calW(\alpha)(T_x\otimes\Id_{\frakM(\alpha)})$, for any $\alpha : x \to y$. Therefore, it induces a map $\phi_T=\oplus_xT_x : M_{\calV} \to M_{\calW}$ such that $\phi_T(v_x)=T_x(v_x)$, for any $v_x\in\calV(x)$, $\forall x \in \Ob\calC$. Since, for any $a\in A$, $v_x\in\calV(x)$ and $\alpha: x \to y$, we have 
\begin{align*}
\phi_T(av_xm_{\alpha}) & =T_y\calV(\alpha)(av_x\otimes m_{\alpha})\\
& =\calW(\alpha)(T_x\otimes\Id_{\frakM(\alpha)})(av_x\otimes m_{\alpha})\\
& =\calW(\alpha)(T_x(av_x)\otimes m_{\alpha})\\
& =\calW(\alpha)(aT_x(v_x)\otimes m_{\alpha})\\
& =a\phi_T(v_x)m_{\alpha},
\end{align*}
the map $\phi_T : M_{\calV} \to M_{\calW}$ is indeed a $A$-$\frakM[\calC]$-bimodule homomorphism. Based on the assignments $\calV\mapsto M_{\calV}$ and $T\mapsto\phi_T$, we obtain a functor 
$$
\iota_*: \Hom_{\Bim_k^{\calC}}(\underline{A},\frakM) \to \Hom_{\Bim_k}(A,\frakM[\calC]).
$$

Thirdly, one can check that these two functors, $\pi_*$ and $\iota_*$, are quasi-inverse to each other. To check $\iota_*\pi_*$, we begin with a $A$-$\frakM[\calC]$-bimodule $M$. It results in an object, a lax transformation, $\calV_M:=\pi_*(M) : \underline{M} \to \underline{\frakM[\calC]} \to \frakM$. On each $x\in\Ob\calC$, it is 
$$
\calV_M(x)= M \otimes_{\frakM[\calC]}\frakM[\calC]1_{\frakM(x)}\cong M1_{\frakM(x)}.
$$
While on each $\alpha : x \to y$, it is the $\frakM[\calC]$-$\frakM(y)$-bimodule homomorphism 
$$
\calV_M(\alpha) : \calV_M(x)\otimes_{\frakM(x)}\frakM(\alpha) \to \calV_M(y),
$$
or more explicitly the natural bimodule homomorphism
$$
\calV_M(\alpha) : M \otimes_{\frakM[\calC]}\frakM[\calC]1_{\frakM(x)}\otimes_{\frakM(x)}\frakM(\alpha) \to M \otimes_{\frakM[\calC]}\frakM[\calC]1_{\frakM(y)},
$$
given by $M1_{\frakM(x)}\otimes_{\frakM(x)}\frakM(\alpha) \to M1_{\frakM(y)}$. It is obvious that
$$
M = \bigoplus_{x\in\Ob\calC} M1_{\frakM(x)} = \bigoplus_{x\in\Ob\calC} \calV_M(x) = M_{\calV_M} = \iota_*\pi_*(M).
$$
Moreover, it is straightforward to verify that this is a bimodule isomorphism, and that $\iota_*\pi_*$ is naturally isomorphic to the identity functor on $\Hom_{\Bim_k}(A,\frakM[\calC])$. We leave $\pi_*\iota_*$ for the interested reader.
\end{proof}

There is a concept of 2-colimit, which we shall not use here. However, an interesting observation is that the Grothendieck construction on a (contravariant) pseudofunctor $\frakF : \calC^{op} \to \Cat$ is the 2-colimit of $\frakF$ \cite[Section 10.2]{JY}, comparable with our construction of the modulated category algebra.

\begin{corollary} Let $A=k$. Then we have an equivalence 
$$
\rMod\frakM[\calC]=\Hom_{\Bim_k}(k,\frakM[\calC]) \to \Hom_{\Bim_k^{\calC}}(\underline{k},\frakM)=\Rep_k(\calC,\frakM).
$$
\end{corollary}

It is well-known that group representations are equivalent to modules of the corresponding group algebra. This was extended to representations of a small category by Mitchell, to representations of modulated quivers by Simson \cite{Si} and to modules of a presheaf of $k$-algebras by \cite{WX}. The above corollary is regarded as a generalization of all these facts.

Now we can present a connection between 2-representations of $\calC$, in the sense of Ganter-Kapranov \cite{GK}, and the representations of a modulated category, traced back to Gabriel \cite{Ga, DR}.

\begin{corollary} Let $\underline{k} : \calC \to \Vec_k$ be the trivial 2-representation of $\calC$ (or the trivial modulation). Then the representations of $(\calC,\underline{k})$ is identified with $k\calC$-modules, where $k\calC$ is the category algebra.

Particularly for a finite group $G$, the representations of $(G,\underline{k})$ are exactly the $kG$-modules.
\end{corollary}

\begin{proof} It follows directly from the above corollary. The category $\Rep_k(\calC,\underline{k})=\Hom(\underline{k},\underline{k})$ can also be computed directly, say, with Definition \ref{repmc}. This the category of all functors $\calC \to \rMod k$, which is equivalent to $\rMod k\calC$.
\end{proof}

The above indicates that the trivial 2-representation $\underline{k}$ of $\calC$ determines all 1-representations of $\calC$. If we consider a modulation as a 2-representation of $\calC$, in the sense of \cite{GK}, then it gives all 1-representations of the algebra $\frakM[\calC]$ (not by decategorification).

Analogously to the main theorem, we have a parallel statement for comodulated categories.

\begin{theorem}\label{CMC} Let $A$ be an object of $\Bim_k$. Soppose $\Ob\calC$ is finite, and $\frakW : \calC \to \Bim_k$ is a contravariant pseudofunctor. Then the 2-limit of $\frakW$ is $(\frakW[\calC],\pi)$, where $\frakW[\calC]$ is the comudulated category algebra, equipped with a functor
$$
\pi_* : \Hom_{\Bim_k}(A,\frakW[\calC]) \to \Hom_{\Bim_k^{\calC}}(\underline{A},\frakW),
$$
which becomes an equivalence for each $A\in\Ob\Bim_k$.
\end{theorem}

We leave the interested reader to formulate its proof and consequences. Finally, the modules of a presheaf of $k$-algebras are the same as the representations of the corresponding comodulated category.

\begin{example} In the example \label{presheaf} for a given presheaf of algebras $\frakR$ we have a modulation and a comodulation on $\calC$: on every $\alpha : x\to y$ the bimodules are $\frakM_{\frakR}(\alpha)=_{\frakR(x)}\frakR(x)_{\frakR(y)}$ and $\frakW_{\frakR}(\alpha)=_{\frakR(y)}\frakR(x)_{\frakR(x)}$, respectively. Then $\frakM_{\frakR}[\calC]$ and $\frakW_{\frakR}[\calC]$ are different algebras, and the second one is the skew category algebra $\frakW_{\frakR}[\calC]=\frakR[\calC]$ that we introduced in \cite{WX}. In fact, $\frakM_{\frakR}[\calC]$ is some sort of a dual version of the skew category algebra.

By the above theorem and the remarks behind Definition \ref{CMR}, $\Rep_k(\calC,\frakW_{\frakR})$ is the same as the category of right $\frakR$-modules. Thus Theorem \ref{CMC} becomes a vast generalization of the main theorem in \cite{WX}! 
\end{example}

\section{Finiteness conditions on modules of a presheaf of algebras}

In \cite{WX}, we considered objects of finite type in the module category $\rMod\frakR$, from the point of view of \cite{Po}. However, these modules have a geometric flavor, and there is a distinct way to introduce modules of finite type, based on slice categories. Here we shall compare these two finiteness conditions and show that they are different.

\subsection{Finiteness conditions on objects of an abelian category} Here we state the finiteness conditions from a categorical viewpoint by Popescu \cite{Po}. 

\begin{definition}
Let $\calA$ be an Ab5 category and $X\in\Ob(\calA)$. A directed set $\{X_{i}\}_{i\in I}$ of subobjects of $X$ is called stationary if there exists an index $i_{0}$ such that $X_{i}\subset X_{i_{0}}$ for any $i\in I$, and it is complete if $\Sigma_{i\in I}X_{i}=X$.
    
If any complete direct set $\{X_{i}\}_{i\in I}$ of subobjects of $X$ is stationary, then $X$ is said to be of finite type.
\end{definition}

It is known that in a module category $\calA=$ Mod-$A$, an object is of finite type if and only if it's a finitely generated $A$-module. 

\subsection{Finiteness conditions on modules of a presheaf of algebras} In the case of Mod-$\frakR$, there is a geometric description of finite type modules, see \cite[18.23.1]{Stack}, which we recall now. 

Note that we will focus on the case where $\calC$ carries the minimal topology, which means that on every object $x\in\Ob\calC$ there is only one covering sieve, the maximal sieve corresponding to the representable functor, or the set of all morphisms ending at $x$. Under the minimal topology, every presheaf is a sheaf. In this situation, the definition is simplified as follows.

\begin{definition} Let $\calV$ be an $\frakR$-module. Then it is of finite type if for every $x\in\Ob\calC$, $\calV|_x$ is a quotient of $\frakR|_x^{n_x}=(\frakR|_x)^{n_x}$ for some positive integer $n_x$.
\end{definition}

Here $\calV|_x$ is the restriction of $\calV$ along the canonical functor $\pi_x : \calC/x \to \calC$, where $\calC/x$ is the \textit{slice category} at $x$. By definition, the objects of $\calC/x$ are of the form $(w,\alpha)$, where $\alpha : w \to x$ is a morphism in $\calC$, and a morphism $(w,\alpha)\to(v,\beta)$ is given by a morphism $\gamma : w\to v$ in $\calC$ such that $\alpha=\beta\gamma$.

\subsection{Comparison of finiteness conditions} We aim to show that the aforementioned finiteness conditions are different on Mod-$\frakR$. There are two canonical right $\frakR$-modules: one is $\frakR$ itself and the other corresponds to the regular $\frakR[\calC]$-module, which as a presheaf $\calP$ of $k$-modules is given by $\calP(x)=\oplus_yk\Hom_{\calC}(x,y)\otimes_k\frakR(x)$. If we denote by $\calP_y$ the $\frakR$-module corresponding to the projective $\frakR[\calC]$-module $1_{\frakR(y)}\frakR[\calC]$ (where $1_{\frakR(y)}$ is an idempotent in the algebra), then it is given by $\calP_y(x)=k\Hom_{\calC}(x,y)\otimes_k\frakR(x)$, $\forall x \in \Ob\calC$.

\begin{proposition} Let $\calC$ be a finite category and $\frakR$ be a presheaf of finite-dimensional $k$-algebra. Then the $\frakR$-modules of finite type are finitely generated $\frakR[\calC]$-modules.
\end{proposition}

\begin{proof} Let $\calV$ be an $\frakR$-module of finite type. Then for every $x\in\Ob\calC$, $\calV|_x$ is a quotient of $\frakR|_x^{n_x}=(\frakR|_x)^{n_x}$ for some positive integer $n_x$. Particularly $\calV|_x(x,1_x)=\calV(x)$ is a quotient of $\frakR|_x^{n_x}(x,1_x)=\frakR(x)^{n_x}$. Write $\phi_x : \frakR(x)^{n_x} \to \calV(x)$ for the surjective homomorphism of right $\frakR(x)$-modules. If $\frakR$ is a presheaf of finite-dimensional $k$-algebras, then every $\calV(x)$ is finite-dimensional, therefore the module $M=\oplus_x\calV(x)$ corresponding to $\calV$ is finite-dimensional, and it is a finitely generated $\frakR[\calC]$-module.
\end{proof}

However, finitely generated $\frakR[\calC]$-modules are not necessarily of finite type $\frakR$-modules. Here is an explicit example.

\begin{example} Let $\calC$ be the category $\xymatrix{x \ar@/^/[r]^{\alpha} \ar@/_/[r]_{\beta} & y}$ and we can define a presheaf of $k$-algebras $\frakR$ on $\calC$ as follows
$$
\xymatrix{
    \frakR(x) & \frakR(y). \ar@/^/[l]^{\frakR(\beta)} \ar@/_/[l]_{\frakR(\alpha)}
}
$$
The regular module of the skew category algebra decomposes as 
$$
\frakR[\calC]=1_{\frakR(x)}\frakR[\calC]\oplus 1_{\frakR(y)}\frakR[\calC].
$$
Consider the $\frakR$-modules $\calP_x$ and $\calP_y$, corresponding to the finitely generated $\frakR[\calC]$-modules $1_{\frakR(x)}\frakR[\calC]$ and $1_{\frakR(y)}\frakR[\calC]$ respectively. Then the preceding decomposition can be expressed as
$$
\calP=\calP_{x}\oplus\calP_{y}=[\xymatrix{\frakR(x) & 0 \ar@/^/[l]^{0} \ar@/_/[l]_{0}}]\oplus[\xymatrix{\frakR(x)\oplus\frakR(x) & \frakR(y) \ar@/^/[l]^{\frakR(\beta)} \ar@/_/[l]_{\frakR(\alpha)}}],
$$ 
where $\frakR(\alpha)$ and $\frakR(\beta)$ map $\frakR(y)$ into the first and second summands $\frakR(x)$, respectively. We have $\calC/x$ as a category with one object $(x,1_x)$ and a single morphism which is the identity. Therefore,
$\calP_{x}|_x=\frakR(x)$. We can write out $\calC/y$ explicitly and then
$$
\xymatrix{ && \frakR(x) &\\
\calP_{x}|_y & = && 0 \ar[lu]_{0} \ar[ld]^{0}\\
 & & \frakR(x) & .
}
$$
In the same way we get $\calP_{y}|_x=\frakR(x)\oplus\frakR(x)$ and
$$
\xymatrix{ & & \frakR(x)\oplus\frakR(x) &\\
\calP_{y}|_y& = & & \frakR(y) \ar[lu]_{\frakR(\alpha)} \ar[ld]^{\frakR(\beta)}\\
 & & \frakR(x)\oplus\frakR(x) & .
}
$$
Now let $\frakR=\underline{k}$ be the constant functor and we consider a morphism $\phi : \underline{k}|_y \to \calP_x|_y$. From the following commutative diagram
$$
\xymatrix{
& & k \ar@/_/[ddd]_{\phi_{1}} & \\
\underline{k}|_y \ar[ddd]_{\phi} & & & k \ar[ul]_{1_{k}} \ar[dl]^{1_{K}} \ar@/^/[ddd]^{0}\\
& = & k \ar@/_/[ddd]_{\phi_{2}} &\\
& & k &\\
\calP_x|_y & & & 0 \ar[ul]_{0} \ar[dl]^{0}\\
& & k &
}
$$
we get $\phi_{1}=\phi_{2}=0$ which forces $\phi=0$. It implies that it is impossible to find an epimorphism $\underline{k}^n|_y \to \calP_x|_y$ for any $n$, and thus $\calP_x$ is not of finite type in the geometric sense. Same can be said for $\calP_y$ by similar calculations.
\end{example}

\end{document}